\documentclass{amsart}

\usepackage{amssymb,amsmath,latexsym,graphics,amsfonts}

\newtheorem{theorem}{Theorem}[section]
\newtheorem{corollary}[theorem]{Corollary}
\newtheorem{lemma}[theorem]{Lemma}

\theoremstyle{definition}
\newtheorem{definition}[theorem]{Definition}
\theoremstyle{definition}
\newtheorem{example}[theorem]{Example}
\newtheorem{remark}[theorem]{Remark}

\numberwithin{equation}{section}

\begin{document}

\title{On contact and symplectic Lie algebroids}

\author[Nazari, Heydari]{Esmail Nazari, Abbas Heydari} 






\address{%
{\bf Esmail Nazari}\\ Department of Mathematics\\
Faculty of Mathematical Sciences
\\Tarbiat Modares University\\Tehran 14115-134, Iran\\
e.nazari@modares.ac.ir}

\address{%
{\bf Abbas Heydari}\\ Department of Mathematics\\
Faculty of Mathematical Sciences
\\ Tarbiat Modares University \\Tehran 14115-134, Iran\\
aheydari@modares.ac.ir}

\date{}

\maketitle
\begin{abstract}
In this paper, we will study compatible triples on Lie algebroids. Using a suitable decomposition for a Lie algebroid, we construct an integrable generalized distribution on the base manifold. As a result, the symplectic form on the Lie algebroid induces a symplectic form on each integral submanifold of the distribution. The induced Poisson structure on the base manifold can be represented by means of the induced Poisson structures on the integral submanifolds. Moreover, for any compatible triple with invariant metric and admissible almost complex structure, we show that the bracket annihilates on the kernel of the anchor map.
\end{abstract}

\section{Introduction}
Symplectic geometry is motivated by the mathematical formalism of classical mechanics. On the other hand, in the last decades, the Lie algebroids have an important place in the context of some different categories in differential geometry and mathematical physics and represent an active domain of research.(\cite{AN},\cite{MB}, \cite{FB}, \cite{CM}) The Lie algebroids, are generalizations of Lie algebras and integrable distributions\cite{KM}. In fact a Lie algebroid is an anchored vector bundle with a Lie bracket on the module of sections and many geometrical notions which involves the tangent bundle were generalized to the context of Lie algebroids.
 The notion of symplectic Lie algebroid is introduced in  \cite{RN}.

The current paper contains four sections. In section two we review basic definitions and facts about Lie algebroids and decompositions of transitive Lie algebroids based on \cite{FB}. Invariant metrics on transitive Lie algebroids are also mentioned in this section. Section 3 includes two parts. The first one is dedicated to symplectic Lie algebroids, almost complex structures  and compatible triples on Lie algebroids. In the second part we have used a compatible triple  to give a decomposition for a Lie algebroid. This decomposition induces an integrable generalized distribution on the base manifold. We investigate the relation between the Poisson structures on the integral submanifolds and Poisson structure on the base manifold. Finally in section 4 we discuss contact Lie algebroids, mainly based on \cite{CIC}. We find coditions for the contact form of the Lie algebroid to induce a Poisson structure on the base manifold.
\section{Lie algebroids}
\begin{definition}
A Lie algebroid $\mathcal{A}$ over a smooth manifold $M$ is a vector bundle $\pi: \mathcal{A} \rightarrow M$ together with a Lie algebra structure $[\, , ]$ on the space $\Gamma(\mathcal{A})$ of sections and a bundle map $\rho : \mathcal{A} \rightarrow TM$ called the anchor map such that
\begin{enumerate}
\item The induced map $\rho : \Gamma(\mathcal{A}) \longrightarrow \mathcal{X}(M)$ is a homomorphism of Lie algebras, that is, $\rho([S_1,S_2]_\mathcal{A})=[\rho(S_1),\rho(S_2)]$ for $S_1,S_2\in\Gamma(\mathcal{A})$.
\item For any sections $S_1,S_2\in\Gamma(\mathcal{A})$ and for every smooth function $f\in C^\infty(M)$
the Leibniz identity $[S_1,fS_2]=f[S_1,S_2]+(\rho(S_1)\cdot f)S_2$ is satisfied.
\end{enumerate}
\end{definition}
The basic example of a Lie algebroid over $M$ is the tangent bundle $TM$ with the identity map as the anchor map.  \\
Lie algebroids can also be smaller or larger than $TM$.  Any integrable distribution of $TM$ is a Lie algebroid with the induced bracket and the inclusion as the anchor map. On the other hand, any Lie algebra $\mathfrak{g} $ is
a Lie algebroid over a point.  \\
An important operator associated with a Lie algebroid $(\mathcal{A},\rho,[\,,])$ over a manifold $M$ is
the exterior derivative $d^\mathcal{A} : \Gamma\big(\wedge^k\mathcal{A^*}\big)\longrightarrow\Gamma \big(\wedge^{k+1}\mathcal{A^*}\big)$ of $\mathcal{A}$ which is defined as follows
$$\begin{array}{ll}
 d^\mathcal{A}(\eta)(S_0,...,S_{k})=&\sum\limits_{i=0}^k(-1)^i\rho(S_i)\cdot\eta(S_1,...,\widehat{S_i},...,S_k)\\ \\
&+\sum\limits_{i<j=1}^k(-1)^{i+j}\eta([S_i,S_j],S_1 ,...,\widehat{S_i},...,\widehat{S_j},...,S_k)
\end{array}$$
for $\eta \in \Gamma(\wedge^k\mathcal{A^*})$ and $S_0,...,S_k \in \Gamma(\mathcal{A})$. It follows that $(d^\mathcal{A})^2=0$.(\cite{CM})\\
Moreover, if $S$ is a section of $\mathcal{A}$, one may introduce, in a natural way, the Lie derivative with respect to $S$, as the operator $\mathcal{L}:\Gamma\big(\wedge^k\mathcal{A^*}\big)\longrightarrow\Gamma \big(\wedge^{k}\mathcal{A^*}\big)$ given by
$$\mathcal{L}_S=i_S\circ\mathrm{d}+\mathrm{d}\circ i_S$$
where $i_S$ is the inner contraction with $S$. Exactly like ordinary manifolds, the usual property $\mathcal{L}_S\circ\mathrm{d}=\mathrm{d}\circ\mathcal{L}_S$ holds here(see\cite{EL}), as well as the relations
$$i_{[S,T]}=[\mathcal{L}_{S},i_{T}]\ \   \ \   \ \    \ \    \ \  [\mathcal{L}_{S},\mathcal{L}_{T}]=\mathcal{L}_{[S,T]}.$$
\begin{example}
	Let $(M,\{\,,\})$ be a Poisson manifold equipped with a bivector $\pi$. Consider the Lie bracket $[\,,]_\pi$ on $\Gamma(T^*M)$ as follows
	$$[\alpha,\beta]=\mathcal{L}_{\pi^\sharp(\alpha)}\beta-\mathcal{L}_{\pi^\sharp(\beta)}\alpha-d(\pi(\alpha,\beta)$$
	for $\alpha,\beta\in \Gamma(T^*M).$ Then $(T^*M,\pi^\sharp,[\,,]_\pi)$ is a Lie algebroid over $M$.\cite{FR}
\end{example}
\begin{definition}
A Lie algebroid $(\mathcal{A},\rho,[\,,])$ is called {\it{transitive}} (respectively {\it regular}) if $\rho$ is surjective (respectively constant rank).
\end{definition}
An immediate consequence of this definition is that, for any $p$ in $M$, there
is an induced Lie bracket say ${[\, , ]}_{p}$ on
$$L_p= Ker(\rho_p) \subset A_p$$
which makes it into a Lie algebra\cite{KM}. For a Lie algebroid $(\mathcal{A}, \rho, [\,,])$ over $M$, the image of $\rho$ defines a smooth integrable generalized distribution in $M$(\cite{MC}). The derived foliation is called the{\it characteristic foliation} of $\mathcal{A}$. Let $N$ be any leaf of the characteristic foliation of $\mathcal{A}$ on $M$, it is easy to see that the bracket of $\mathcal{A}$ deduce the bracket on the space of section of the restriction $\mathcal{A}_N$ of $\mathcal{A}$ to $N$. Then $\rho_{|_{\mathcal{A}_N}} :\mathcal{A}_N \rightarrow TN$ is a transitive Lie algebroid(\cite{L}).
\begin{definition}\cite{EL}
	Let $(\mathcal{A},[\,,]_\mathcal{A},\rho_\mathcal{A})$ and  $(\mathcal{B},[\,,]_\mathcal{B},\rho_\mathcal{B})$ be two Lie algebroids over smooth manifolds $M$ and $M^\prime$, respectively. A vector bundle map $(\Phi,\phi):\mathcal{A}\rightarrow\mathcal{B}$ is called a {\it morphism of Lie algebroids} if for every $\alpha\in\Gamma(\wedge^k\mathcal{B^*})$ we have
	\begin{equation}\label{dphi}	\Phi^*(\mathrm{d}_\mathcal{B}\alpha)=\mathrm{d}_\mathcal{A}\big(\Phi^*(\alpha)\big)
	\end{equation}
	where  $\Phi^*:\wedge^k\mathcal{B}^*\longrightarrow\wedge^k\mathcal{A}^*$ is defined by
	$$\Phi^*(\alpha)_p(S_1,...,S_n)=\alpha_{\phi(p)}\big(\Phi(S_1),...,\Phi(S_n)\big)$$
	for $\alpha\in\Gamma(\wedge^k\mathcal{B}^*)$, $p\in M$ and $S_1,...,S_n\in \mathcal{A}_p$.
\end{definition}
It is easy to see that $(\Phi,\phi)$ is a morphism of Lie algebroids if for every $S\in\Gamma\mathcal{A}$, $\rho_\mathcal{B}(\Phi(S))=\phi_*(\rho_\mathcal{A}(S))$ and equation (\ref{dphi}) holds for every 1-form section of $\mathcal{B}$. Moreover, if $\phi$ is diffeomorphism then $\Phi$ maps any section of $\mathcal{A}$ to a section of $\mathcal{B}$. In this case, $\Phi$ is a morphism of Lie algebroids if $\rho_\mathcal{B}\circ\Phi=\phi_*\circ\rho_\mathcal{A}$ and $\Phi([S_1,S_2]_\mathcal{A})=[\Phi(S_1),\Phi(S_2)]_\mathcal{B}$ for any $S_1,S_2\in\Gamma\mathcal{A}$(\cite{EL}).\\
For a transitive Lie algebroid, $L=ker \rho$ is a bundle of Lie algebras
\cite{KM}. Suppose $(\mathcal{A},\rho,[\,,])$ is a transitive Lie algebroid, then a vector bundle
map $\lambda : TM\rightarrow \mathcal{A}$ such that $\rho \circ\lambda=1_{TM}$ is a splitting of $(\mathcal{A},\rho,[\,,])$ i,e., we can decompose $\mathcal{A}$ to $ E+L$ of vector bundles, where $E = \lambda(TM)$. \\
Fix a $\lambda : TM\rightarrow \mathcal{A}$ splitting of $\rho$. The map $\lambda$ defines a linear connection on $L$, called an adjoint connection as follows.
\begin{align*}
\nabla^\lambda : \mathcal{X}(M)\times\Gamma(L)\rightarrow\Gamma(L)\\
\nabla^\lambda_XT:=[\lambda(X),T]
\end{align*}
The 2-form $\Omega^\lambda$ in $A^2(M,L)$ is defined as follows
$$2\Omega^\lambda(X,Y)=[\lambda(X),\lambda(Y)]-\lambda([X,Y])$$
and is called the curvature 2-form.
The Lie bracket on $\Gamma(\mathcal{A})$ with respect to the decomposition $\mathcal{A}=E+L$ is written as follows
$$[\![\lambda(X)+S,\lambda(Y)+T]\!]=\lambda([X,Y])+\nabla^\lambda_XT-\nabla^\lambda_YS+[S,T]+\Omega(X,Y).$$
Conversely if $L$ is a bundle of Lie algebras, and $\nabla$ is a connection on $L$ that preserves the Lie bracket and the curvature of $\nabla$ is in the form $[2\Omega(X,Y),S]$ for $S\in \Gamma(L)$ and some $\Omega \in A^2(M,L)$, then we can make $TM+L$ into a transitive Lie algebroid by defining a Lie bracket on $\Gamma(TM+L)$ as follows
$$[\![X+S,Y+T]\!]=[X,Y]+\nabla_XT-\nabla_YS+[S,T]+\Omega(X,Y).$$
For more details, see\cite{FB}. So we have the following theorem:
\begin{theorem}\cite{FB}\label{split}
All transitive Lie algebroids have the above structure.
\end{theorem}
\begin{definition}\cite{FB}
A Riemannian metric $g$ on $\mathcal{A}$ is said to be {\it invariant} if all adjoint connections of $\mathcal{A}$ preserve the restriction $g_L$ of $g$ to $L$, i.e., for every $\lambda$ and $X \in \mathcal{X}(M)$, $\nabla^\lambda_X g_L=0$.
\end{definition}
Having an invariant Riemannian metric $g$ one can write
$$g([S_1,S_2],S_3)=g(S_1,[S_2,S_3])$$
for $S_1,S_2,S_3\in\Gamma(L)$.
\begin{theorem} \cite{FB}  \label{DL}
If $g$ is an invariant metric on $\mathcal{A}$ and $\nabla$ is the Levi-Civita connection of $\mathcal{A}$ then
\begin{equation}
\nabla_{S_1}S_2=\frac{1}{2}[S_1,S_2]
\end{equation}
for every $S_1,S_2\in \Gamma(L)$.
\end{theorem}
\section{Compatible triples on Lie algebroids}
\subsection{Symplectic Lie algebroids}
\begin{definition}\cite{CI} An {\it{almost complex structure $J_\mathcal{A}$}} on $(\mathcal{A},\rho,[\,,])$ is an endomorphism $J_\mathcal{A}:\Gamma(\mathcal{A})\rightarrow \Gamma(\mathcal{A})$, over the identity, such that $J_\mathcal{A}^2=-{\rm id_{\Gamma(\mathcal{A})}}$. A Lie algebroid $(\mathcal{A},\rho,[\,,],J_\mathcal{A})$ endowed with such a structure is called an almost complex Lie algebroid.
\end{definition} 
\begin{definition}
We call an almost complex structure $J_\mathcal{A}$ on $\mathcal{A}$, {\it admissible} (or called admissible with respect to $J_M$) if there exists an almost complex structure $J_M$ on M such that
$$\rho \circ J_\mathcal{A}=J_M \circ \rho.$$
\end{definition}
\begin{remark}
\begin{enumerate}
\item We will use the notion of integrability of almost complex structures on Lie algebroids as Popescu developed in \cite{CI}.
\item If $J_\mathcal{A}$ is admissible with respect to $J_M$, and $N_{J_\mathcal{A}}$ and $N_{J_M}$ are Nijenhuis tensors of $(\mathcal{A},J_\mathcal{A})$ and $(M,J_M)$ respectively, we have
\begin{equation}\label{NNM}
\rho(N_{J_\mathcal{A}}(a,b))=N_{J_M}(\rho(a),\rho(b))
\end{equation}
for every $a,b$ in $\Gamma(\mathcal{A})$.(\cite{CI})
\end{enumerate}
\end{remark}
\begin{theorem}
An almost complex structure $J_\mathcal{A}$ on transitive Lie algebroid $\mathcal{A}$ is admissible if and only if $J_\mathcal{A}(L)\subseteq L$.
\end{theorem}
\begin{proof}
Let  $\lambda : TM\longrightarrow \mathcal{A}$ be a splitting of $\rho$ and $E:=\lambda(TM)$, i.e., $\mathcal{A}=E+L$. For $X\in TM$ put
$${J_M}(X):=\rho\big(J_\mathcal{A}(\lambda(X))\big).$$
We show that $J_M$ is an almost complex structure on $M$. To prove this we need to show that $J_M^2=-id_{TM}$. For $X \in \mathcal{X}(M)$
\begin{equation}\label{A1}
J_M^2(X)=\rho\Big(J_\mathcal{A}\big(\lambda(\rho\big(J_\mathcal{A}(\lambda(X))\big))\big)\Big)=\rho\Big(J_\mathcal{A}\big(\big(J_\mathcal{A}(\lambda(X))\big)^E\big)\Big)
\end{equation}
where $\big(J_\mathcal{A}(\lambda(X))\big)^E$ is the $E$-part of $J_\mathcal{A}(\lambda(X))$ with respect to the given decomposition of $\mathcal{A}$. On the other hand,
$J_\mathcal{A}(\lambda(X))=J_\mathcal{A}(\lambda(X))^E+J_\mathcal{A}(\lambda(X))^L$. Thus by applying $J_\mathcal{A}$ we have $J_\mathcal{A}(J_\mathcal{A}(\lambda(X))^E)=-X-J_\mathcal{A}(J_\mathcal{A}(\lambda(X))^L)$. Again by applying $\rho$ we have $\rho(J_\mathcal{A}(J_\mathcal{A}(\lambda(X))^E))=-\rho(\lambda(X))=-X$.
So by ($\ref{A1}$), $J_M^2=-Id_{TM}$, i.e.,
$J_M$ is an almost complex structure on $M$ that clearly satisfies $J_M \circ \rho=\rho \circ J_\mathcal{A}$.\\
The converse is obvious.
\end{proof}
\begin{example}
Let $(M,J_M,g_M)$ be an almost Hermitian manifold with the Levi- Civita connection $\nabla$. We denote the induced connection on $L(TM)$ again by $\nabla$. This connection preserves Lie bracket of $L(TM)$. We also denote the curvature of $\nabla$ on $L(TM)$ by $R^\prime$. So for $T \in L(TM)$, we have $R^\prime (X,Y)T=[R(X,Y),T]$. Consequently, we can define  an algebroid structure on $TM+L(TM)$ by the following Lie bracket on $\Gamma(TM+L(TM))$.(\cite{FB})\\
For every $X,Y\in\mathcal{X}(M)$ and $T,S$ in $\Gamma(L(TM))$
$$[\![X+T,Y+S]\!]=[X,Y]+\nabla_XS-\nabla_YT+[T,S]+R(X,Y).$$
Now we can define an almost complex structure on $\Gamma(L(TM))$, induced by the almost complex structure $J_M$ on $M$, as
$$\begin{array}{ll}
 J_{L(TM)} :& \Gamma(L(TM))\longrightarrow\Gamma(L(TM))\\ \\
&J_{L(TM)}(T)(X):=T(J_M(X)) \ \  \ \  \ \  (X \in \mathcal{X}(M)).\\
\end{array}$$
Using these two almost complex structures, we define an almost complex structure on $TM+L(TM))$ as $J(X+T)=J_M(X)+J_{T(M)}(T)$ for $X\in\mathcal{X}(M)$ and $T\in\Gamma(L(TM))$. Clearly $J$ is admissible with respect to $J_M$.
It is easy to see that $g_{L(TM)}(T,S):=trace(TS^T)$ is a Riemannian metric on $L(TM)$, compatible with $J_{L(TM)}$. Thus $(TM+L(TM),J,g)$ is an almost Hermitian Lie algebroid where
$$g(X+T,Y+S)=g_M(X,Y)+g_{L(TM)}(T,S).$$
Moreover, $(M,J_M,g_M)$ is a Hermitian manifold if and only if $(TM+L(TM),J,g)$ is a Hermitian Lie algebroid. In fact, if $(M,J_M,g_M)$ is Hermitian, then the Nijnhuis tensor of $J$ can be calculated in the following three conditions:
\begin{enumerate}
\item If $T\in\Gamma(L(TM))$ and $X\in\mathcal{X}(M)$ then
$$\begin{array}{lll}\vspace{0.2cm}
N(T,X)(Y)&=&([J(T),J(X)]-[T,X]-J[J(T),X]-J[T,J(X)])(Y)\\ \vspace{0.2cm}
&=&(\nabla_{J(X)}J(T) - \nabla_XT - J(\nabla_XJ(T)) - J(\nabla_{J(X)}T))(Y)\\ \vspace{0.2cm}
&=&\nabla_{J_M(X)}T(J_M(Y))-T(J_M(\nabla_{J_M(X)}Y))-\nabla_XT(Y)\\ \vspace{0.2cm}
&&+T(\nabla_XY)+\nabla_XT(Y)+T(J_M(\nabla_XJ_M(Y)))\\ \vspace{0.2cm}
&&-\nabla_{J_M(X)}T(J_M(Y))+T(\nabla_{J_M(X)}J_M(Y))\\ \vspace{0.2cm}
&=&T\big(J_M(\nabla_XJ_M)(Y)\big)+T((\nabla_{J_M(X)}J_M)(Y))\\
&=&0.
\end{array}$$
\item If $T,S\in\Gamma(L(TM))$ then
$$\begin{array}{lcl}\vspace{0.2cm}
N(T,S)(X)&=&([J(T),J(S)]-[T,S]-J[J(T),S]-J[T,J(S)])(X)\\ \vspace{0.2cm}
&=&T(J_M(S(J_M(X))))-S(J_M(T(J_M(X))))-T(S(X))\\ \vspace{0.2cm}
&&+S(T(X))-T(J_M(S(J_M(X))))-S(T(X))\\ \vspace{0.2cm}
&&+T(S(X))+S(J_M(T(J_M(X))))\\
&=&0.
\end{array}$$
\item
If $X,Y\in\mathcal{X}(M)$ then
\begin{equation}\label{NXY}
\begin{array}{lll}\vspace{0.2cm}
N(X,Y)&=&[J_M(X),J_M(Y)]-[X,Y]-J_M[J_M(X),Y]\\\vspace{0.2cm}
&&-J_M[X,J_M(Y)]+R(J_M(X),J_M(Y))-R(X,Y)\\
&&-J_M(R(X,J_M(Y)))-J_M(R(J_M(X),Y)).
\end{array}
\end{equation}
\end{enumerate}
The first line in the above equation is clearly zero. So we need to show that the second line is zero.\\
Integrability of $J_M$ leads to the following calculation 
$$\begin{array}{lcl}\vspace{0.2cm}
 J_M(R(X,Y)Z)&=&J_M(\nabla_X\nabla_YZ)-J_M(\nabla_Y\nabla_XZ)-J_M(\nabla_{[X,Y]}Z)\\ \vspace{0.2cm}
&=&\nabla_X\nabla_YJ_M(Z)-\nabla_Y\nabla_XJ_M(Z)-\nabla_{[X,Y]}J_M(Z)\\
&=&R(X,Y)J_M(Z)
\end{array}$$
For $X,Y,Z\in\mathcal{X}(M)$.
Using the property of curvature tensor $R$  for $U,V\in\mathcal{X}(M)$ we have
\begin{align*}
 g_M(R(X,J_M(Y))U,V)=&g_M(R(U,V)X,J_M(Y)\\
 =&-g_M(R(U,V)J_M(X),Y)\\
 =&-g_M(R(J_M(X),Y)U,V).
\end{align*}
One can easily show that $R(X,J(Y))=-R(J(X),Y)$. By ($\ref{NXY}$) we see that  $N(X,Y)=0$ for every $X,Y$ in $\mathcal{X}(M)$, i.e., $(TM+L(TM),J,g)$ is Hermitian.\\
Conversely if  $(TM+L(TM),J,g)$ is a Hermitian Lie algebroid then by equation (\ref{NXY}) we have $N_{J_M}=0$ and so $(M,J_M,g_M)$ is a Hermitian manifold.
\end{example}
\begin{definition}\cite{RN}
A {\it{symplectic Lie algebroid}} is a Lie algebroid $(\mathcal{A},\rho,[,])$ together with a closed and non-degenerate 2-form $\omega$ on $\mathcal{A}$.
\end{definition}
\begin{remark}
For a symplectic Lie algebroid $(\mathcal{A},\omega)$ and every smooth function $f$ on $M$ there exists a unique  section $a_f\in\Gamma(\mathcal{A})$ such that
$$\mathrm{d}_\mathcal{A}f(b)=\omega(a_f,b) \ \ \ \ \ \ \ \ (b\in \Gamma(\mathcal{A})).$$
This is called the {\it{Hamiltonian section}} of $f$. Using this, one can define a Poisson structure on $M$ as follows(\cite{RN}):
$$\{f,g\}:=\omega(a_f,a_g).$$
\end{remark}
There is no analogue to the Lie's third theorem in the case of Lie algebroids, since not every Lie algebroid can be integrated to a global Lie groupoid, although there are local versions of this result, (see \cite{MC}).
\begin{theorem}
For a symplectic Lie algebroid $(\mathcal{A},\rho,[\,,],\omega)$ if $\omega|_L$ is nondegenerate, then $\mathcal{A}$ is integrable.
\end{theorem}
\begin{proof}
Let $\mathcal{N}$ be a leaf of the characteristic foliation of $\mathcal{A}$ then $(\mathcal{A}_\mathcal{N},\rho_\mathcal{N},[\,,]_\mathcal{N},\omega_\mathcal{N})$ is a symplectic Lie algebroid over $\mathcal{N}$. Suppose that $E$ is the symplectic complement of $L_\mathcal{N}$, i.e., $E=L^{\omega}_\mathcal{N}$. Since $\omega|_L$ is nondegenerate,  $L\cap L^{\omega}={0}$. So $\mathcal{A}_\mathcal{N}=E+L_\mathcal{N}$ is a decomposition of $\mathcal{A}_\mathcal{N}$. We claim that $E$ is closed under the bracket. Considering $\Omega$ as the curvature 2-form with respect to the given decomposition of $\mathcal{A}_\mathcal{N}$, we have
$$\begin{array}{cll}
0&=&\mathrm{d}_{\mathcal{A}_\mathcal{N}}\omega_\mathcal{N}(a,b,s)=\rho_\mathcal{N}(a).\omega_\mathcal{N}(b,s)+\rho_\mathcal{N}(b).\omega_\mathcal{N}(a,s)-\rho_\mathcal{N}(s).\omega_\mathcal{N}(a,b)\\ \\
&&+\omega_\mathcal{N}([a,s]_\mathcal{N},b)-\omega_\mathcal{N}([a,b]_\mathcal{N},s)-\omega_\mathcal{N}([b,s]_\mathcal{N},a])\\ \\
&=&\omega_\mathcal{N} (\Omega(a,b),s),\\
\end{array}$$
for all $a,b \in  \Gamma E$ and $s \in \Gamma L_\mathcal{N}$.\\
This means that $\Omega =0$. Thus $\Gamma E$ is closed under the bracket.
Hence by Corollary 5.2, of \cite{MC}, $\mathcal{A}$ is integrable.
\end{proof}
\begin{example}
 Let $(\mathfrak{g},\omega)$ be a symplectic Lie algebra, i.e., $\mathfrak{g}$ is a Lie algebra and $\omega$ is a non-degenerate 2- form on it and
 $$\omega([S_1,S_3],S_2)-\omega([S_1,S_2],S_3)-\omega([S_2,S_3],S_1)=0 \ \ \  \ \ \ (S_1,S_2,S_3 \in \mathfrak{g}),$$
 then $(M\times\mathfrak{g},0,[\,,]_{\mathfrak{g}},\omega)$ is a symplectic Lie algebroid.
 \end{example}
\begin{example} Let $(L,\omega_L)$ is a symplectic Lie algebra bundle over a symplectic manifold $(M,\omega_M)$. Consider the flat connection $\nabla$ on $L$ that preserves $\omega_L$. Then $(L,\nabla)$ together with the zero curvature 2-form construct a transitive Lie algebroid, $TM + L$. Put
\begin{align*}
\omega: \ \ &\Gamma(TM+ L)\times\Gamma(TM+ L)\rightarrow C^\infty(TM+ L)\\
&\omega(X+S,Y+T)=\omega_M(X,Y)+\omega_L(S,T)
\end{align*}
clearly $\omega$ is nondegenerate. Moreover, $\omega$ is closed, since  for every $X\in \mathcal{X}(M)$ and $S,T \in \Gamma(L)$ we have
$$\begin{array}{lcl}\vspace{0.2cm}
\mathrm{d}^{(TM+L)}\omega(X,S,T)&=&\rho(X).\omega(S,T)-\rho(S).\omega(X,T)+\rho(T).\omega(X,T)\\ \vspace{0.2cm}
&&-\omega([X,S],T)+\omega([X,T],S)-\omega([T,S],X)\\\vspace{0.2cm}
&=&\rho(X).\omega_L(S,T)-\omega_L(\triangledown_XS,T)-\omega_L(S,\triangledown_XT)\\\vspace{0.2cm}
&=&(\triangledown_X\omega_L)(S,T)\\
&=&0.\\
\end{array}$$
Other cases $\big($i.e. $d\omega(X,Y,Z)=0$ and $d\omega(X,S,T)=0$ and $d\omega(S,T_1,T_2)=0\big)$ are trivial.
Thus $(TM+L,\omega)$ is a symplectic Lie algebroid over $M$.
\end{example}
\begin{definition}
Over a symplectic Lie algebroid $(\mathcal{A},\omega)$, the triple $(\omega,J,g)$ is called compatible if $J$ is an almost complex structure and $g$ is a Riemannian metric on $\mathcal{A}$ such that
$$g(J(S),J(T))=g(S,T) \  \ and \ \omega(S,T)=g(S,J(T))\ \  \ \ \ \ \ \ (S,T \in \Gamma(\mathcal{A})).$$
\end{definition}
For such a triple, we also have(\cite{CI})
\begin{equation}\label{nablaN}
2g((\nabla_aJ)b,c)=g(N(b,c),J(a)),
\end{equation}
where $\nabla$ is the Levi-Civita connection on $\mathcal{A}$ and $N$ is the Nijenhuis tensor of $J$.
\subsection{Decomposition of Lie algebroids with compatible triple}
Suppose that $(\omega,J,g)$ be a compatible triple on a Lie algebroid $\mathcal{A}$ over $M$. For every $p\in M$ put $L^1_p=L_p\cap L_p^\omega$ and $L^2_p=(L^1_p)^\perp$ under the $g_L$ restriction of $g$ on $L$. Note that $\omega|_{L^2_p}$ is nondegenerate. Putting $E^1_p=J(L^1_p)$, for $J(T_p)\in E^1_p$ and $S_p\in L_p$ we have $$g(S_p,J(T_p))=\omega(S_p,T_p)=0.$$
Thus $E^1_p$ is perpendicular to $L_p$. Finally taking $E^2_p=(L_p+E^1_p)^\omega$ one can see that $E^2_p\cap( L_p+E^1_p)=0$, since for $a_p\in E^2_p\cap( L_p+E^1_p)$ there is $S_p\in L_p$ and $T_p\in L^1_p$ such that $a_p=S_p+J(T_p)$. Clearly
$$0=\omega(a_p,T_p)=\omega(S_p+J(T_p),T_p)=g(T_p,T_p).$$
Thus $T_p=0$, and so  $a_p=S_p$. Since $a_p\in L_p^\omega$, we have $S_p\in L_p^1$. On the other hand, $S_p \in (E^1_p)^\omega$, hence
 $$0=\omega(J(S_p),S_p)=g(S_p,S_p),$$
thus $a_p=0$. Therefore, $(E^1_p+E^2_p)\oplus L_p$ is a decomposition of $\mathcal{A}_p$, i.e., $\mathcal{A}_p=(E^1_p+E^2_p) \oplus (L^1_p+L^2_p)$. Note that these sets (including $E^1,E^2,L^1,L^2$) may not be subbundles or distributions, and they may have not constant rank. However, under certain circumstances, some combinations of these sets are subbundles. For instance, if $\mathcal{A}$ is transitive then $E^2+L^1$ is a sub Lie algebroid.\\
The restriction of $[\,,]$ to $L^1_p$ is zero. In fact $\omega$ is closed so, for every $S_p,T_p\in L^1_p$ and $Z_p\in L_p$ we have
$$0=d\omega(S_p,T_p,Z_p)=-\omega([S_p,T_p],Z_p).$$
Therefore, $[S_p,T_p]\in L_p^1$.  Again for $Z_p\in L_p^1$ we have
$$0=d\omega(S_p,T_p,J(Z_p))=-\omega([S_p,T_p],J(Z_p))=g([S_p,T_p],Z_p)$$
thus $[S_p,T_p]\in L_p^2$. Therefore, $[S_p,T_p]=0$.\\
Moreover, $\rho(E^2+L^1)$ is an integrable generalized distribution on $M$. If $\mathcal{N}$ is a leaf of $\mathcal{A}$ then $(E^2+L^1)_{|_\mathcal{N}}$ is a Lie sub algebroid of $\mathcal{A}_{|_\mathcal{N}}\rightarrow \mathcal{N}$, because for every $a,b\in\Gamma(E^2+L^1)$ and $S\in\Gamma(\mathcal{A})$ such that $S_p\in L$, we have
$$0=d\omega(a,b,S)(p)=-\omega([a,b],S)(p),$$
which means that
\begin{equation}\label{[XY]}
[a,b]\in \Gamma(E^2+L^1).
\end{equation}
The above equation shows that $E^2+L^1$ is closed under the bracket. Now we put $\Lambda:=\rho(E^2+L^1)$ and call it the {\it symplectic generalized distribution} of $(\mathcal{A},\omega)$ on $M$. For $x\in M $ if $\mathcal{O}_x$ is the integral submanifold of $\Lambda$ at $x$, we can define $\omega_{\mathcal{O}_x}$ as a nondegenerate 2-form on $\mathcal{O}_x$ by 
\begin{align*}
\omega_{\mathcal{O}_x}: \ \ \mathcal{X}(\mathcal{O}_x)\times \mathcal{X}(\mathcal{O}_x)&\longrightarrow C^\infty(\mathcal{O}_x)\\
\omega_{\mathcal{O}_x}(X,Y)(p)&:=\omega(a,b)(p),
\end{align*}
where $a,b\in \Gamma(E^2+L^1)$ satisfy $\rho(a_p)=X_p,\rho(b_p)=Y_p$ for every $p$ in $\mathcal{O}_x$.
Clearly $\omega_{\mathcal{O}_x}$ is a 2-form on $\mathcal{O}_x$. Since $\omega_{|_{E^2}}$ is nondegenerate, so is  $\omega_{\mathcal{O}_x}$. Moreover, for $p\in \mathcal{O}_x$ and $X,Y,Z \in \mathcal{X}(\mathcal{O}_x)$, there exist $a,b,c$ in $\Gamma\mathcal{A}$ such that $a_p,b_p, c_p\in E^2_p$ and $\rho(a_p)=X_p$, $\rho(b_p)=Y_p$ and $\rho(c_p)=Z_p$. Therefore,
\begin{align*}
d\omega_{\mathcal{O}_x}(X,Y,Z)(p)=&X.\omega_{\mathcal{O}_x}(Y,Z)(p)-Y.\omega_{\mathcal{O}_x}(X,Z)(p)-Z.\omega_{\mathcal{O}_x}(X,Y)(p)\\
&-\omega_{\mathcal{O}_x}([X,Y],Z)(p)+\omega_{\mathcal{O}_x}([X,Z],Y)(p)-\omega_{\mathcal{O}_x}([Y,Z],X)(p)\\
=&\rho(a).\omega(b,c)(p)-\rho(b).\omega(a,c)(p)-\rho(c).\omega(a,b)(p)\\
&-\omega\big([a,b]_p-[a,b]^{L^1}_p,c_p\big)+\omega\big([a,c]_p-[a,c]_p^{L^1},b_p\big)\\\
&-\omega\big([b,c]_p-[b,c]_p^{L^1},a_p\big)\\
=&d^\mathcal{A}\omega(a,b,c)(p)+\omega\big([a,b]_p^{L^1},c_p\big)-\omega\big([a,c]_p^{L^1},b_p\big)\\
&+\omega\big([b,c]_p^{L^1},a_p\big)\\
=&0,
\end{align*}
which means that $\omega_{\mathcal{O}_x}$ is closed and so is a symplectic form on $\mathcal{O}_x$.
\begin{theorem}
 Let $(\omega,J,g)$ be a compatible triple on Lie algebroid $\mathcal{A}$ and $\Lambda$ be the symplectic generalized distribution of $(\mathcal{A},\omega)$ on $M$. Then for every integral submanifold $\mathcal{O}$ of $\Lambda$ we have
 $$\{f,g\}_{|_\mathcal{O}}=\{f_{|_\mathcal{O}},g_{|_\mathcal{O}}\}_{\mathcal{O}},$$
 where $\{\,,\}$ is the Poisson structure induced by $\omega$ and $\{\,,\}_\mathcal{O}$ is the Poisson structure induced by $\omega_\mathcal{O}$ on $\mathcal{O}$.
 \end{theorem}
 \begin{proof}
For $f,g\in C^\infty(M)$ let $a_f$ and $a_g$ be the Hamiltonian sections of $f,g$ with respect to $\omega$ and $X_{f_{|_\mathcal{O}}},X_{g_{|_\mathcal{O}}}$ be the Hamiltonian vector fields of $f_{|_\mathcal{O}}$ and $g_{|_\mathcal{O}}$ with respect to $\omega_{\mathcal{O}}$. Then
$$\rho(a_f(x))=X_{f_{|_\mathcal{O}}}(x)\,, \ \ \ \ \rho(a_g(x))=X_{g_{|_\mathcal{O}}}(x) \ \ \ \ \ \ (x\in\mathcal{O}).$$
Thus for every $x$ in $\mathcal{O}$ we have
$$\{f_{|_\mathcal{O}},g_{|_\mathcal{O}}\}_{\mathcal{O}}(x)=\omega_\mathcal{O}\big(X_{f_{|_\mathcal{O}}}(x),X_{g_{|_\mathcal{O}}}(x)\big)=
\omega\big(a_f(x),a_g(x)\big)=\{f,g\}_{\omega_{|_\mathcal{O}}}(x),$$
which completes the proof.
\end{proof}
\begin{corollary}
With the above notations, if $f$ is constant on $\mathcal{O}_x$ for some $x\in M$, then
$$\{f,g\}(y)=0,$$
for every $g\in C^\infty(M)$ and $y\in \mathcal{O}_x$.
\end{corollary}
Suppose that $(\omega,J,g)$ be a compatible triple on Lie algebroid $\mathcal{A}$ such that $J$ preserves $L^2$. Then for any $a_p\in E_p^2$ and $S_p\in L_p^2$ we have
$$0=\omega(J(S_p),a_p)=g(S_p,a_p),$$
which means that $E^2_p$ is perpendicular to $L_p^2$. Our claim is that $E^2$ is invariant under $J$. To prove this, for $a$ in $E^2_p$ consider the equation
$$J(a)=\big( J(a)\big)^{E^2_p}+\big(J(a)\big)^{L_p^2},$$
where$E^2_p$ is perpendicular to $L_p^1+E_p^1$ and $J$ is compatible with $g$. Thus $J(a)$ has no component through $L_p^1$ and $E_p^1$. Therefore,
$$\begin{array}{lcl}\vspace{0.2cm}
 0&=&\omega\Big(\big(J(a)\big)^{L^2_p},a\Big)\\ \vspace{0.2cm}
&=&\omega\Big(J\big(\big(J(a)\big)^{L^2_p}\big),J(a)\Big)\\ \vspace{0.2cm}
&=&\omega\Big(J\big(\big(J(a)\big)^{L^2_p}\big),\big( J(a)\big)^{E^2_p}+\big(J(a)\big)^{L^2_p}\Big)\\
&=&\parallel J\big((J(a))^{L^2_p}\big)\parallel.
\end{array}$$
Hence $\big(J(a)\big)^{L^2_p}$ is zero, for every $a$ in $E^2_p.$\\
Next, for an integral submanifold $\mathcal{O}$ of $\Lambda$, we define $g_\mathcal{O}$ and $J_\mathcal{O}$ as follows
$$J_\mathcal{O}(X)=\rho(J(a))\ \   \ \  \ \  , \ \ g_\mathcal{O}(X,Y)=g(a,b),$$
where $X,Y \in T\mathcal{O}$ and $a,b\in E^2$ such that $\rho(a)=X$ and $\rho(b)=Y$.\\
Clearly $J_\mathcal{O}$ is an almost complex structure and $g_\mathcal{O}$ a Riemannian metric on $M$ such that the triple $(\omega,J_\mathcal{O},g_\mathcal{O})$ is compatible on $\mathcal{O}$.
We have proved the following result:
\begin{theorem}
Let $(\omega,J,g)$ be a compatible triple on $\mathcal{A}$ together with decomposition $(E^1_p+E^2_p)+L_p$ for $\mathcal{A}_p$, such that $J$ preserves $L^2$. Then for every integral submanifold $\mathcal{O}$ of $\Lambda=\rho(E^2+L^1)$, $(\omega_\mathcal{O},J_\mathcal{O},g_\mathcal{O})$ is a compatible triple on $\mathcal{O}$.
\end{theorem}
\begin{corollary}
Let $(\omega,J,g)$ be a compatible triple on a transitive Lie algebroid $\mathcal{A}$ such that $J$ is admissible. Then $\omega_{|_L}$ is nondegenerate, $L^1$ and $E^1$ are null,  and $\Lambda=TM$. Also $(M,\omega_M)$ is a symplectic manifold and the triple $(\omega_M,g_M,J_M)$ is compatible. Moreover the Poisson structure induced by $\omega$ is equal to the Poisson structure induce by $\omega_M$ on $M$.
\end{corollary}

\begin{theorem}\label{E}
Let $(\omega,J,g)$ be a compatible triple on a Lie algebroid $\mathcal{A}$ with the decomposition $\mathcal{A}_p=(E^1_p+E^2_p)+(L^1_p+L^2_p)$. Then for any leaf $\mathcal{N}$ of the characteristic foliation of $\mathcal{A}$, the Lie algebroid $\big((E^2+L^1)_\mathcal{N},\rho_{|_{(E^2+L^1)_\mathcal{N}}},[\,,]_{|_{(E^2+L^1)_\mathcal{N}}}\big)$  and $(T^*\mathcal{N},\pi^*,[\,,]_{T^*\mathcal{N}})$ are isomorphic over $\mathcal{N}$.
\end{theorem}
\begin{proof}
Note that $(\mathcal{A}_\mathcal{N},\rho_\mathcal{N},[\,,]_\mathcal{N})$ is a transitive Lie algebroid over $\mathcal{N}$. Also $L_\mathcal{N}=Ker(\rho_{\mathcal{N}})$ is a Lie algebra subbundle of $(\mathcal{A}_\mathcal{N},\rho_\mathcal{N},[\,,]_\mathcal{N})$. Since $E^2_\mathcal{N}+L^1_\mathcal{N}=L_\mathcal{N}^{\omega_\mathcal{N}}=-J(L_\mathcal{N}^\perp)$, where $\omega_\mathcal{N}$ is the restriction of $\omega$ to $\mathcal{A}_\mathcal{N}$, $(E^2+L^1)_\mathcal{N}$ is a vector subbundle of $\mathcal{A}_\mathcal{N}$. Moreover, by (\ref{[XY]}), $\Gamma(E^2+L^1)_\mathcal{N}$ is closed under the bracket and thus inherits the Lie algebroid properties.\\ Looking at $\mathcal{A}_\mathcal{N}$ as $L_\mathcal{N}+L_\mathcal{N}^\perp$, one can consider the vector bundle map
\begin{align*}
J^\perp : &\mathcal{A}_\mathcal{N}\longrightarrow \mathcal{A}_\mathcal{N}\\
&a\mapsto J(a)^{L_\mathcal{N}^\perp}.
\end{align*}
It is easy to see that $J^\perp\big((E^2+L^1)_\mathcal{N}\big)=(E^1+E^2)_\mathcal{N}$. Thus $(E^1+E^2)_\mathcal{N}$ is smooth and so $(E^1+E^2)_\mathcal{N}+L_\mathcal{N}$ is a decomposition of $\mathcal{A}_\mathcal{N}$. Let $\lambda : T\mathcal{N}\rightarrow \mathcal{A}_\mathcal{N}$ be the corresponding splitting with respect to this decomposition given by Theorem \ref{split}. We define
\begin{align*}
\psi : &(E^2+L^1)_\mathcal{N}\longrightarrow T^*\mathcal{N}\\
&a\longmapsto -\lambda^*(i_a\omega)
\end{align*}
Since $\omega$ is nondegenerate, $\psi$ is one to one. Thus $\psi$ is an isomorphism of vector bundles.\\
 To complete the proof we need to prove that $\psi$ preserves the bracket and anchor maps. To do this, we take an arbitrary $f\in C^\infty(\mathcal{N})$ and denote its Hamiltonian section on $\mathcal{N}$ by $a_f$. Then
\begin{align*}
\big(\rho(a_f)+\pi^*(df)\big).g=&dg(a_f)+dg(\pi^*(df))\\
=&\omega(a_g,a_f)+\pi(df,dg)\\
=&\{g,f\}-\{g,f\}\\
=&0,
\end{align*}
i.e., $\rho(a_f)=-\pi^*(df)$. Thus $a_f+\lambda(\pi^*(df))\in L_\mathcal{N}$. Now for $a\in (E^2+L^1)_\mathcal{N}$ and $f\in C^\infty(\mathcal{N})$, we have
\begin{align*}
\big(\rho(a)-\pi^*(\psi(a))\big).f=&\omega(a_f,a)-\pi(\psi(a),df)\\
=&\omega(a_f,a)+\psi(a)(\pi^*(df))\\
=&\omega(a_f,a)-\omega(a,\lambda(\pi^*(df)))\\
=&\omega(a_f+\lambda(\pi^*(df)),a)\\
=&0.
\end{align*}
Thus $\rho(a)=\pi^*(\psi(a))$, i.e., $\psi$ preserves the anchor maps.\\
On the other hand, for $a,b\in \Gamma\big((E^2+L^1)_\mathcal{N}\big)$ and $X\in \mathcal{X}(\mathcal{N})$ we have
\begin{align*}
[\psi(a),\psi(b)](X)=&\mathcal{L}_{\pi^*(\psi(a))}\psi(b)(X)-\mathcal{L}_{\pi^*(\psi(b))}\psi(a)(X)-d(\pi(\psi(a),\psi(b))(X)\\
=&\mathcal{L}_{\rho(a)}\psi(b)(X)-\mathcal{L}_{\rho(b)}\psi(a)(X)-d(\pi(\psi(a),\psi(b))(X)\\
=&\rho(a).\psi(b)(X)-\psi(b)([\rho(a),X])-\rho(b).\psi(a)(X)\\
&\psi(a)([\rho(b),X])-d(\psi(b)(\pi^*(\psi(a))(X)\\
=&-\rho(a).\omega(b,\lambda(X))+\omega(b,\lambda([\rho(a),X]))\\
&+\rho(b).\omega(a,\lambda(X))-\omega(a,\lambda([\rho(b),X]))+X.\omega(b,\lambda(\rho(a)))\\
=&-d\omega(a,b,\lambda(X))-\omega([a,b],\lambda(X))\\
=&-\lambda^*(i_{[a,b]}\omega)(X)\\
=&\psi([a,b])(X),
\end{align*}
which means that $\psi$ preserves the bracket and so it is a Lie algebroid isomorphism.
\end{proof}

\begin{theorem}\label{L0}
If $(\omega,J,g)$ is a compatible triple on transitive Lie algebroid $\mathcal{A}$ such that $g$ is compatible and $J$ is admissible then
\begin{enumerate}
\item The bracket of every two sections of L is zero.
\item $(\mathcal{A},\omega,J,g)$ is K$\ddot{a}$hler if and only if $(M,\omega_M,J_M,g_M)$ is K$\ddot{a}$hler.
\end{enumerate}
\end{theorem}
\begin{proof}
(1) $\omega$ is closed. Therefore, for $S_1,S_2,S_3 \in \Gamma(L)$ we have
$$0=d\omega(S_1,S_2,S_3)=-\omega([S_1,S_2],S_3)+\omega([S_1,S_3],S_2)-\omega([S_2,S_3],S_1).$$
Since $(\omega,J,g)$ are compatible we can rewrite the above equation as
$$-g([S_1,S_2],JS_3)+g([S_1,S_3],JS_2)-g([S_2,S_3],JS_1)=0.$$
Since $g$ is invariant we have
\begin{equation}\label{g}
-g([S_1,S_2],JS_3)-g([S_1,JS_2],S_3)-g([JS_1,S_2],S_3)=0.
\end{equation}
This shows that
$$-[S_1,S_2]-J([S_1,JS_2])-J([JS_1,S_2])=0.$$
Using the above equation, one can calculate $N(S_1,S_2)$ as follows
\begin{equation}\label{N1}
 N(S_1,S_2)=[JS_1,JS_2]-[S_1,S_2]-J[JS_1,S_2]-J[S_1,JS_2]=[JS_1,JS_2].
\end{equation}
Replacing $S_1$  by $JS_1$, in (\ref{g}) we have
\begin{equation*}
g(J[JS_1,S_2],S_3)-g([JS_1,JS_2],S_3)+g([S_1,S_2],S_3)=0,
\end{equation*}
which means that
$$-[JS_1,JS_2]+[S_1,S_2]+J[JS_1,S_2]=0.$$
Therefore,
\begin{equation}\label{N2}
 N(S_1,S_2)=-J[S_1,JS_2].
\end{equation}
By applying (\ref{N2}) and (\ref{DL}) in (\ref{nablaN}) we have
\begin{equation}
\begin{array}{rcl}\vspace*{0.2cm}
0&=&2g\big((\nabla_{S_1}J)S_2,S_3\big)-g\big(N(S_2,S_3),JS_1\big)\\ \vspace*{0.2cm}
&=&g([S_1,JS_2],S_3)-g(J[S_1,S_2],S_3)+g(J[S_2,JS_3],JS_1)\\ \vspace*{0.2cm}
&=&g([S_1,JS_2],S_3)-g(J[S_1,S_2],S_3)-g(J[S_1,S_2],S_3)\\
&=&g([S_1,JS_2]-2J[S_1,S_2],S_3),
\end{array}
\end{equation}
i.e., for every $S_1,S_2\in \Gamma(L)$
$$[S_1,JS_2]=2J[S_1,S_2].$$
Applying this to (\ref{N1}) and (\ref{N2}) we get $N(S_1,S_2)=[JS_1,JS_2]=-4[S_1,S_2]$. On the other hand, $N(S_1,S_2)=-J[S_1,JS_2]=2[S_1,S_2]$. 
Therefore, $[S_1,S_2]=0$ and $N(S_1,S_2)=0$.\\
(2) Now look at $\mathcal{A}$ as $L^\perp+L$. By the proof of Theorem \ref{E}, $L^\perp$ is  closed under the restriction bracket. Suppose that $(M,\omega_M,J_M,g_M)$ is K$\ddot{a}$hler, therefore, $N_{J_M}=0$, i.e.,
$$N_{J_M}(X,Y)=0 \quad \quad \ \ (X,Y \in \mathcal{X}(M)).$$
By (\ref{NNM})
$$N_J(a,b)=0 \ \   \ \   \  \   \ \   \ \  (a,b \in \Gamma(L^\perp)).$$
We have just proved, in the previous part, that $N(S,T)=0$  for every $S,T\in\Gamma(L)$. To complete the proof, it suffices to show that $N(S,a)=0$ for every $S\in \Gamma(L)$ and $a\in\Gamma(L^\perp)$. Using (\ref{nablaN}) we have
$$g(N(S,a),T)=-g((\nabla_{JT}J)(S),a)=0 \ \  \  \   \ \  (T\in \Gamma(L)).$$
On the other hand, $g(N(S,a),b)=0$ for every $b\in\Gamma(L^\perp)$ and so $N(S,a)=0$. Thus $(\mathcal{A},J,\omega,g)$ is K$\ddot{a}$hler. The converse is trivial.
\end{proof}
\begin{theorem}
 On a transitive K$\ddot{a}$hler Lie algebroid $(\mathcal{A},\omega,J,g)$ if $g$ is invariant then the restriction of the bracket on $\Gamma(L)$ is zero.
\end{theorem}
\begin{proof}
 Fix a point $p\in M$. Using the decomposition mentioned at the beginning of the section, one can write $\mathcal{A}_p=(E^1_p+E^2_p)+(L^1_p+L^2_p)$. Since  $(\mathcal{A},\omega,J,g)$ is K$\ddot{a}$hler, $\nabla\omega=\nabla J=0$, where $\nabla$ is the Levi-Civita connection of $g$. Now for $S,T\in L^2_p$ and $Z\in L^1_p$ we have
 \begin{align*}
 0=&(\nabla_S\omega)(T,Z)\\
 =&\omega(\nabla_ST,Z)+\omega(T,\nabla_SZ)\\
 =&\omega(\frac{1}{2}[S,T],Z)+\omega(T,\frac{1}{2}[S,Z])\\
 =&\omega(T,\frac{1}{2}[S,Z]).\\
 \end{align*}
 This means that
 \begin{equation}\label{L1L2L1}
 [L^1_p,L^2_p]\subseteq L^1_p.
 \end{equation}
 Using the above equation and the fact that $g$ is invariant, for $S,T\in L^2_p$ and $Z\in L^1_p$ we have
 \begin{equation}
 g([S,T],Z)=g(S,[T,Z])=0,
 \end{equation}
 and so
\begin{equation}\label{L2L2L2}
[L^2_p,L^2_p]\subseteq L^2_p.
\end{equation}
Moreover for $S,T,Z\in L^2_p$
\begin{align*}
0=&d\omega(S,T,Z)\\
=&-\omega([S,T],Z)+\omega([S,Z],T)-\omega([T,Z],S)\\
=&-\omega([S,T],Z)-2(\nabla_Z\omega)(S,T)\\
=&-\omega([S,T],Z),\\
\end{align*}
i.e., $[L^2_p,L^2_p]\subseteq L^1_p$. Hence  by (\ref{L2L2L2}) we have $[L^2_p,L^2_p]=0$.\\
Furthermore, we know that $[L^1_p,L^1_p]=0$. Thus for $S,T\in L^1_p$ and $Z\in L^2_p$ we have
$$0=g([S,T],Z)=g(S,[T,Z]),$$
i.e., $[L^1_p,l^2_p]\subseteq L^2_p$. Using (\ref{L1L2L1}) one can easily see that $[L^1_p,L^2_p]=0$. This proves that the restriction of the bracket on $L$ is zero.\\
\end{proof}
\section{Contact Lie algebroid}
\begin{definition}\cite{CIC}
 Let $\mathcal{A}$ be a Lie algebroid of rank $2n+1$ over a smooth m-dimensional  manifold $M$. A 1-form $\eta\in \Gamma(\mathcal{A}^*)$ is called {\it contact} if $\eta\wedge(d\eta)^n\neq 0$. In this case, $(\mathcal{A},\eta)$ is called a {\it contact Lie algebroid}.
 \end{definition}
 For a contact Lie algebroid $(\mathcal{A},\eta)$ there exists a unique section $\xi\in \Gamma \mathcal{A}$ called the {\it Reeb section} such that
 $$\eta(\xi)=1 ,   \ \    \ \    \ \   i_\xi\mathrm{d}\eta=0.$$
 A triple $(\varphi,\xi,\eta)$ is called an {\it almost contact structure} on $\mathcal{A}$ if $\varphi$ is a $(1,1)$ tensor section of $\Gamma(\mathcal{A}\otimes \mathcal{A}^*)$, $\xi\in \Gamma\mathcal{A}$, $\eta \in \Gamma(\mathcal{A}^*)$ and
 $$\eta(\xi)=1, \ \ \ \ \ \ \ \varphi^2=-id_\mathcal{A}+\eta\otimes\xi.$$
 Moreover $(\mathcal{A},\varphi,\xi,\eta)$ is called an {\it almost contact Lie algebroid}.\\
 A Riemannian metric $g$ on $\mathcal{A}$ is said to be compatible with an almost contact structure $(\varphi,\xi,\eta)$ if
 $$g(\varphi(S),\varphi(T))=g(S,T)-\eta(S)\eta(T) \ \ \ \ \ \ (S,T\in \Gamma\mathcal{A}).$$
 In this case, $(\mathcal{A},\varphi,\xi,\eta,g)$ is called an
 {\it almost contact Riemannian Lie algebroid}.\\
 For an almost contact Riemannian Lie algebroid $(\mathcal{A},\varphi,\xi,\eta,g)$ if
 $$\mathrm{d}\eta(S,T)=g(S,\varphi(T))$$
 then $\eta$ is a contact form, $\xi$ is the Reeb section and $(\mathcal{A},\varphi,\xi,\eta,g)$ is called a {\it contact Riemannian Lie algebroid}.\cite{CIC}\\
 Let $(\mathcal{A},\varphi,\xi,\eta)$ be an almost contact Lie algebroid. Then $\mathrm{D}:= Ker(\eta)$ is a vector subbundle of $\mathcal{A}$ of rank $2n$.  If $\eta$ is contact  then $\mathrm{d}\eta_{|_\mathrm{D}}$ is nondegenerate.\\
Let $(\mathcal{A},\eta)$ be a contact Lie algebroid, consider the vector bundle morphism
$$\psi:\mathcal{A}\longrightarrow\mathcal{A}^*$$
$$S\longmapsto i_S\mathrm{d}\eta.$$
It is easy to see that
$$\mathcal{A}^*=Im(\psi) \oplus\langle\eta\rangle.$$
Thus for $f\in C^\infty(M)$, there exists $S_f\in\Gamma(\mathcal{A})$ and $h\in C^\infty(M)$ such that
\begin{equation}\label{df}
\mathrm{d}f=i_{S_f}\mathrm{d}\eta+h\eta.
\end{equation}
In fact
$$h=\rho(\xi)\cdot f$$
If $S_f,\bar{S}_f$ satisfy (\ref{df}), there exist $k\in C^\infty(M)$ such that
\begin{equation}\label{s1s2}
S_f-\bar{S}_f=k\xi.
\end{equation}
Putting $a_f:=S_f-\eta(S_f)\xi$ one can easily see that $a_f$ satisfies the (\ref{df}). Also 
$a_f$ is independent of the choise of $S_f$. In fact, if $S_f,\bar{S}_f$ satisfy (\ref{df}), then by (\ref{s1s2}) we have
$$S_f-\eta(S_f)\xi=\bar{S}_f-\eta(\bar{S}_f)\xi.$$
We call $a_f$, the {\it Hamiltonian section} of $f$.

 Now we can define
$$\{\,,\}: C^\infty(M)\times C^\infty(M)\longrightarrow C^\infty(M)$$
$$(f,g)\longmapsto \mathrm{d}\eta(a_f,a_g).$$
In fact for $f,g\in C^\infty(M)$
\begin{equation}\label{[f,g]}
\{f,g\}=\rho(a_g)\cdot f=-\rho(a_f)\cdot g
\end{equation}
\begin{theorem}\label{pisson}
	On a contact Lie algebroid $(\mathcal{A},\eta)$ if $\rho(\xi)=0$ then the above bracket is a Poisson structure on $M$.
\end{theorem}
\begin{proof}
	$\{\,,\}$ is $\mathbb{R}$-bilinear and skew-symmetric. Now if $a_f,a_g,a_h$ are Hamiltonian sections of $f,g,h\in C^\infty(M)$, respectively, then
	\begin{align*}
	\mathrm{d}(fg)=&f\mathrm{d}g+g\mathrm{d}f\\
	=&fi_{a_g}\mathrm{d}\eta+gi_{a_f}\mathrm{d}\eta\\
	=&i_{(fa_g+ga_f)}\mathrm{d}\eta.\\
	\end{align*}
	Moreover, $\eta(fa_g+ga_f)=0$. Thus $fa_g+ga_f$ is the Hamiltonian section of $fg$. Hence
	\begin{align*}
    	\{fg,h\}=&\mathrm{d}\eta(fa_g+ga_f,a_h)\\
	   =&f\mathrm{d}\eta(a_g,a_h)+g\mathrm{d}\eta(a_f,a_h)\\
	   =&f\{g,h\}+g\{f,h\},
	\end{align*}
	i.e., $\{\,,\}$ satisfies the product rule.\\
	Furthermore,
	\begin{align*}
	i_{[a_f,a_g]}\mathrm{d}\eta=&[\mathcal{L}_{a_f},i_{a_g}]\mathrm{d}\eta\\
	=&\mathrm{d}i_{a_f}i_{a_g}\mathrm{d}\eta+i_{a_f}\mathrm{d}i_{a_g} \mathrm{d}\eta-i_{a_g}\mathrm{d}i_{a_f}\mathrm{d}\eta-i_{a_g}
i_{a_f}\mathrm{d}\mathrm{d}\eta\\
	=&\mathrm{d}\eta(a_f,a_g)\\
	=&\{f,g\}.
	\end{align*}
	Thus $[a_f,a_g]-\eta([a_f,a_g])\xi$ is the Hamiltonian section of $\{f,g\}$.\\
	Now we can prove the Jacobi identity. We use (\ref{[f,g]}) as follows
	$$\begin{array}{lllll}

	\{f,\{g,h\}\}&=&-\{f,\rho(a_g).h\}&=&\rho(a_f).\rho(a_g).h\\ \\
	\{g,\{h,f\}\}&=&\{g,\rho(a_f).h\}&=&-\rho(a_g).\rho(a_f).h\\ \\
	\{h,\{f,g\}\}&=&\rho(a_{\{f,g\}}).h&=&\rho([a_f,a_g]).h
	\end{array}$$
	i.e.,
	$$\{f,\{g,h\}\}+\{g,\{h,f\}\}+\{h,\{f,g\}\}=0.$$
\end{proof}

\begin{lemma}\label{L sub D}
 If $(\mathcal{A},\eta)$ is a contact Lie algebroid such that $ L_p\subseteq\mathrm{D}_p$ for some $p\in M$, then $L_p=0$.
\end{lemma}
\begin{proof} For $S_p\in L_p\subseteq\mathrm{D}_p$, one can extend $S_p$ to a section $S$ of $\mathrm{D}$. Now for $T\in \Gamma(\mathrm{D})$ we have
 $$d\eta_p(S_p,T_p)= \rho(S_p).\eta(T)-\rho(T_p).\eta(S)-\eta([S,T]_p)=0,$$
where the last identity is a result of ideal property of $\Gamma L$. Since $\eta$ is contact, $\mathrm{d}\eta_{|_\mathrm{D}}$ nondegenerate, Thus $S_p=0$, i.e., $L_p=0$.
\end{proof}
\begin{lemma}\label{xiL}
For an almost contact structure $(\varphi,\xi,\eta)$ on a Lie algebroid $\mathcal{A}$ if $\eta$ is contact and $\varphi(L_p)\subseteq L_p$, for some $p\in M$, then $L_p=0$ or $\rho(\xi_p)=0$.
\end{lemma}
\begin{proof}
Choose $p\in M$ such that $\varphi(L_p)\subseteq L_p$. If $L_p\neq0$ then by Lemma \ref{L sub D} there exists $S_p\in L_p$ such that $\eta_p(S_p)\neq0$. Thus
$$\eta_p(S_p)\xi_p=\varphi^2(S_p)+S_p.$$
Since $\varphi$ preserves $L_p$, $\varphi^2(S_p)\in L_p$, and so $\xi_p\in L_p$.
\end{proof}
Let $(\mathcal{A},\varphi,\xi,\eta,g)$ be a Riemannian contact Lie algebroid such that $\rho$ is surjective and $\varphi$ preserves $L$. If $L= 0$ then $\mathcal{A}=TM$ and our Riemannian contact Lie algebroid reduces to an ordinary contact metric manifold. So let $L\neq 0$. Since $\xi\in \Gamma L$, by (\ref{xiL}), we may write
$$\mathrm{D}=(\mathrm{D}\cap L)\oplus L^\perp.$$
Moreover, for $S\in \Gamma L^\perp$ and $T\in \Gamma L$, 
\begin{equation}
\mathrm{d}\eta(S,T)=g\big(S,\varphi(T)\big)=0.
\end{equation}
Let $\lambda:TM\rightarrow\mathcal{A}$ be the splitting map corresponding to $\mathcal{A}=L^\perp\oplus L$. Putting $\omega:=\lambda^*(\mathrm{d}\eta)$, one can see that $\omega$ is closed and non-degenerate, i.e., $(M,\omega)$ is a symplectic manifold. Clearly $\varphi$ preserves $L^\perp$. Therefore, $\varphi$ and $g$ induce an almost complex structure $J_M$ on $\mathcal{A}$ (with the decomposition $L+L^\perp$) and a Riemanninan metric $g_M$ on $M$, respectively. Thus, $(\omega,J_M,g_M)$ is a compatible triple on $M$. The interesting  point is that, the Poisson structure induced by Theorem \ref{pisson} coincides with the Poisson structure induced by $\omega$.


\begin{thebibliography}{99}

\bibitem{AN} P. Antunes, J. M. Nunes da Costa, Hypersymplectic structures with torsion on Lie algebroids, \emph{J. Geom. Phys.}, \textbf{104} (2016)  39--53.

\bibitem{MB} M. Boucetta, Riemannian Geometry of Lie Algebroids, \emph{Journal of the Egyptian Mathematical Society}, \textbf{19} (2011)  57--70.

\bibitem{MC}M. Crainic, R. Fernandes, Integrability of Lie brackets, \emph{Ann. of Math.}, \textbf{157} (2003) 575--620.

\bibitem{EL}E. de Leon, J. C. Marrero, E. Martinez, Lagrangian submanifolds and dynamics on Lie algebroids. \emph{J. Phys. A}, \textbf{38} (2005) 241--308.

\bibitem{FB}Gh. Fasihi Ramandi, N. Boroojerdian, Forces Unification in The Framework of Transitive
Lie Algebroids,  \emph{Int. J. Theor. Phys.}, \textbf{54} (2015) 1581--1593.

\bibitem{L} R. L. Fernandes, Lie Algebroids, Holonomy and Characteristic Classes. \emph{Advances in Mathematics}, \textbf{170}, (2002)  119-–179.

\bibitem{CI}C. Ida, P. Popescu, On almost complex Lie algebroids, \emph{Mediterr. J. Math.}, \textbf{13}, (2016) 803--824.

\bibitem{CIC}C. Ida, P. Popescu, Contact structures on Lie algebroids, Preprint available to  arXiv:1507.01110v1 [math.DG].

\bibitem{CM}D. Iglesias, J. Marrero, D. Marti
n de Diego, E. Martinez, E, Padron, Reduction of Symple ctic Lie Algebroids by a Lie Subalgebroid and a Symmetry Lie Group, \emph{SIGMA}, \textbf{3} (2007) 049, 28 pages.

\bibitem{KM}K. Mackenzie, General Theory of Lie Groupoids and Lie Algebroids, \emph{Cambridge University Press}, 2005.

\bibitem{RN} R. Nest, B. Tsygan, Deformations of symplectic Lie algebroids, deformations of holomorphic symplectic structures, and index theorems \emph{Asian J. Math.}, \textbf{5} (2001), 599-–635.

\bibitem{FR} F. Rui Loja, Connections in Poisson geometry. I. Holonomy and invariants, \emph{J. Differential Geom}, \textbf{54} (2000) 303--365.

\end{thebibliography}
\end{document}